\documentclass[11pt,english,spanish]{article}
\usepackage{amsmath}
\usepackage[T1]{fontenc}
\usepackage[letterpaper]{geometry}
\usepackage{amsmath}
\usepackage{amssymb}
\usepackage{comment}
\usepackage{graphicx}
\usepackage{epic}
\usepackage{pstricks}
\usepackage{pst-node}
\usepackage{pstricks-add}

 \RequirePackage{lineno} 

\usepackage{color}

\def\BBox{\kern  -0.2cm\hbox{\vrule width 0.2cm height 0.2cm}}

\newtheorem{lem}{Lemma}
\newtheorem{thm}{Theorem}

\newtheorem{defi}{Definition}

\begin{document}
  
\title{Transitive oriented $3$--Hypergraphs of cyclic orders}
\author{Natalia Garcia-Colin\!\! \thanks{%
pmnatz@gmail.com, Instituto de Matem\'{a}ticas}, Amanda Montejano\!\! 
\thanks{%
montejano.a@gmail.com, Facultad de Ciencias }, Luis Montejano\!\! \thanks{%
luis@matem.unam.mx, Instituto de Matem\'{a}ticas, Supported by
PAPIIT-M\'exico, project IN101912}, \\
Deborah Oliveros\!\! \thanks{%
dolivero@matem.unam.mx, Instituto de Matem\'{a}ticas, Supported by
PAPIIT-M\'exico, project IN101912} \\
{\small Universidad Nacional Aut\'{o}noma de M\'{e}xico, M\'{e}xico} }
\maketitle

\begin{abstract}
In this paper we introduce the definition of transitivity for oriented $3$%
--hypergraphs in order to study partial and complete cyclic orders. This
definition allow us to give sufficient conditions on a partial cyclic
order to be totally extendable. Furthermore, we introduce the $3$-hypergraph
associated to a cyclic permutation and characterize it in terms of cyclic
comparability $3$-hypergraphs.
\end{abstract}


\section{Introduction}

A \emph{partial order} of a set $X$ is a reflexive, asymmetric and transitive binary relation on $X$. When the relation is total the order is called a \emph{linear order}, since it captures the idea of ordering the elements of $X$ in a line. In contrast,  to capture the idea of ordering a set $X$ in a circle, it is well known that  a ternary relation is needed. 

A partial order can be graphically represented  by a transitive oriented graph. In particular, the oriented graph associated to a linear order is a transitive oriented tournament. There is no obvious way for representing graphically a partial cyclic order. In \cite{Alles1991}, the authors introduce certain subclass of cyclic orders that can be represented by means of oriented graphs. However, in the literature there is no custom manner to represent any partial cyclic order.

In this paper we introduce a definition for oriented $3$--hypergraph and a notion of transitivity,   in order to represent partial cyclic orders. Just as it occurs with linear orders and transitive tournaments, with our definition, it happens that the transitive oriented $3$--hypertournament is unique.  Among the study of intersection graphs, the class of permutation graphs has received lots of attention (see, for instance, \cite{Golumbic1980} and references therein). In particular, an old result by Pnueli, Lempel and Even~\cite{Pnueli1971} characterizes permutation graphs in terms of comparability graphs. In this work we define the $3$--hypergraph associated to cyclic permutation, and  cyclic comparability $3$--hypergraphs. 

Using these notions we extend the result mentioned above, by characterizing cyclic permutation $3$--hypergraphs in terms of cyclic comparability $3$--hypergraphs (Theorem~\ref{thm:comparability}). In order to prove Theorem~\ref{thm:comparability} we explore first its oriented version (Theorem\ref{thm:permutation}).

One of the most important problems studied within the study of cyclic ordersis the extendability of cyclic orders to complete (or total) cyclic orders (see \cite{Alles1991,Megiddo1976}). This is, when the orientation of each triplet in the cyclic order can be derived from a total cyclic order. Examples of cyclic orders that are not extendable are well known~\cite{Alles1991}. In this context, we exhibit a class of cyclic orders which are totally extendable (Theorem~\ref{thm:main_nes}).

\section{Transitive Oriented $3$--Hypergraphs and cyclic orders}

\label{sec:cyclic}

Let $X$ be a set of cardinality $n$. A \emph{partial cyclic order} of $X$ is a ternary relation $T\subset X^{3}$ which is \emph{cyclic}: $(x,y,z)\in T\Rightarrow (y,z,x)\in T$; \emph{asymmetric}: $(x,y,z)\in T\Rightarrow(z,y,x)\not\in T$; and \emph{transitive}: $(x,y,z),(x,z,w)\in T\Rightarrow(x,y,w)\in T$. If in addition $T$ is \emph{total}: for each $x\neq y\neq z\neq x$, either $(x,y,z)\in T$ or $(z,y,x)\in T$, then $T$ is called a \emph{complete cyclic order}. Partial cyclic orders has been studied in several papers, see \cite{Megiddo1976, Alles1991, Novak1982, Novotny1983} for excellent references.

A \emph{cyclic ordering} of $X$ is an equivalence class, $[\phi ]$, of the set of linear orderings  (i.e. bijections $\phi :[n]\rightarrow X$, where $[n]$ denotes the set  $\{ 1,2,\dots n\}$) with respect to the \emph{cyclic equivalence relation} defined as: $\phi \sim \psi $, if and only if there exist $k\leq n$ such that $\phi (i)=\psi (i+k)$ for every $i\in \lbrack n]$, where $i+k$ is taken (mod $n$). Note that complete cyclic orders and cyclic orderings correspond to the same concept in different contexts. We use the word \textquotedblleft order\textquotedblright \  to refer to a ternary relation, and the word \textquotedblleft ordering\textquotedblright \  to refer to the cyclic equivalence class.

For the reminder of this paper we will denote each cyclic ordering $[\phi]$ in cyclic permutation notation, $(\phi(1)\, \phi(2)\, \ldots\,\phi(n))$.  For example, there are two different cyclic orderings of $X=\{u,v,w\}$, namely $(u\, v\, w)$ and $(u\, w\, v)$, where $(u\, v\, w)=(v\, w\, u)=(w\,u\, v)$ and $(u\, w\, v)=(v\, u\, w)=(w\, v\, u)$.

\subsection{Transitive and Self Transitive oriented 3-hypergraphs}

We use the standard definition of $3$--hypergraph, to be a pair of sets
 $H=(V(H),E(H))$ where $V(H)$ is the vertex set of $H$, and
 the edge set of $H$ is $E(H)\subseteq {{V(H)}\choose{3}}$.

\begin{defi}
Let $H$ be a $3$--hypergraph. An \emph{orientation} of $H$ is an assignment of exactly one of the two possible cyclic orderings  to each edge. An orientation of a $3$--hypergraph is called an \emph{oriented $3$--hypergraph}, and we will denote the oriented edges by $O(H)$.
\end{defi}

\begin{defi} \label{def:transitive} 
An oriented $3$--hypergraph $H$ is \emph{transitive} if whenever $(u\, v\, z)$ and $(z\, v\, w) \in O(H)$ then $(u\, v\, w) \in O(H)$ (this implies $(u\, w\, z) \in O(H)$).
\end{defi}

Note that every $3$-subhypergraph of a transitive oriented $3$-hypergraph is transitive, fact that we will use throughout the rest of the paper. Also, there is a natural correspondence between partial cyclic orders and transitive oriented $3$--hypergraphs.

A transitive oriented $3$--hypergraph, $H$, with $E(H)={\binom{{V(H)}}{{3}}}$ is called a \emph{$3$--hypertournament}.  Let $TT_n^{3}$ be the oriented $3$--hypergraph with $V(TT_n^{3})=[n]$ and $E(H)= {\binom{{[n]}}{{3}}}$, where the orientation of each edge is the one induced by the cyclic ordering $(1\, 2\,\dots\, n)$. Clearly $TT_{n}^{3}$ is a transitive $3$--hypertournament on $n$ vertices. It is important to note that every transitive $3$--hypertournament on $n$ vertices is isomorphic to $TT_{n}^{3}$ which allows us  to hereafter refer to \emph{the} transitive  $3$--hypertournament on $n$ vertices.  This fact, is implied in the literature regarding cyclic orders, and that is indeed the reason why total, cyclic, asymmetric and transitive ternary relations are called complete cyclic orders (or circles as in \cite{Alles1991}).

Given a (non oriented) $3$--hypergraph the complement is naturally defined. For an oriented $3$--hypergraph it is not clear how to define its oriented complement. However, for  $H$ a spanning oriented subhypergraph of $TT_n^3$, we define its \emph{complement} as the oriented $3$--hypergraph $ \overline{H}$ with $V( \overline{H})=V(TT_n^{3})$ and $O(\overline{H})=O(TT_n^3)\setminus O(H)$.

\begin{defi}
An oriented $3$--hypergraph $H$ which is a spanning subhypergraph of $TT_n^3$ is called \emph{self-transitive} if it is transitive and its complement is also transitive.
\end{defi}

The following lemma that we will refer to it as the   \emph{eveness property} is not difficult to proof and it is left to the reader.

\begin{lem}\label{eveness} 
Let $H$ be a self transitive $3$--hypergraph, then all its induced $3$--subhypergraphs with four vertices necessarily have an even number of hyperedges.
\end{lem}

\subsection{The oriented $3$--hypergraph associated to a cyclic permutation}

\label{sec:cyclic_permutation}

In resemblance to the graph associated to a linear permutation, we will now define the $3$--hypergraph associated to a cyclic permutation and study some of its properties.

A \emph{cyclic permutation} is a cyclic ordering of $[n]$. This is, an equivalence class $[\phi]$ of the set of bijections $\phi:[n] \rightarrow [n]$, in respect to the cyclic equivalence relation.  As mentioned before, a cyclic permutation $[\phi]$ will be denoted by $(\phi(1)\, \phi(2)\,\ldots\,\phi(n))$, and its \emph{reversed permutation} $[\phi^{\prime }]$ corresponds to $(\phi(n)\, \phi(n-1)\, \ldots\,\phi(1))$.

Let $[\phi]$ be a cyclic permutation. Three elements $i, j, k \in [n]$, with  $i< j< k$, are said to be in \emph{clock-wise order} in respect to $[\phi]$ if there is $\psi\in [\phi]$ such that $\psi^{-1}(i)<\psi^{-1}(j)<\psi^{-1}(k)$; otherwise the elements $i,j,k$ are said to be in \emph{counter-clockwise order} with respect to $[\phi]$.

\begin{defi}
The \emph{oriented $3$--hypergraph $H_{[\phi ]}$ associated to a cyclic permutation} $[\phi ]$ is the hypergraph with vertex set $V(H_{[\phi ]})=[n]$ whose edges are the triplets $\{i,j,k\}$, with $i<j<k$, which are in clockwise order in respect to $[\phi ]$, and whose edge orientations are induced by $[\phi ]$.
\end{defi}

It can be easily checked that for any cyclic permutation $[\phi]$, the associated oriented $3$--hypergraph $H_{[\phi]}$ is a transitive  oriented $3$--hypergraph naturally embedded in $TT_{n}^{3}$.  Moreover, the complement of  $H_{[\phi ]}$ is precisely $H_{[\phi ^{\prime }]}$, where $[\phi ^{\prime }]$ is the reversed cyclic permutation of $[\phi ]$. So, for every cyclic permutation $[\phi ]$, the $3$--hypergraph $H_{[\phi ]}$ is self-transitive.

A transitive $3$--hypergraph $H$ is called an \emph{cyclic permutation $3$--hypergraph} if there is a cyclic permutation $[\phi]$ such that $H_{[\phi]} \cong H$. Next we shall prove that the self--transitive property characterizes the class of oriented cyclic permutation $3$--hypergraphs.  For a vertex $v \in V(H)$ of an oriented $3$--hypergraph $H$, the \emph{link} of $v$, denote by $link_H (v)$, is the oriented graph with vertex and arc sets equal to: $V(link_H (v))= V(H)\setminus \{ v\}$, and $O(link_H (v))= \{ (uw) | \: \exists \: (v\,u\,w) \in O(H)\}$.

\begin{lem}
\label{link} Let $H$ be a self--transitive $3$--hypergraph. Then,  $link_H (v)$  is a self transitive oriented graph, for any $v\in V(H)$.
\end{lem}

\begin{proof}
Let $(v\, v_i\, v_j)$ and $(v\, v_j\, v_k) \in O(H)$, by transitivity $(v\, v_i\, v_k) \in O(H)$. So that, if $link_H(v)$ contains arcs $(v_i \, v_j)$ and $(v_j \, v_k)$, then the transitivity of $H$ implies $(v_i\, v_k)$ is also an arc of $link_H(v)$.
\end{proof}

\begin{thm}
\label{thm:permutation} $H$ is an oriented cyclic permutation $3$--hypergraph if and only if $H$ is self transitive.
\end{thm}

\begin{proof}

If  there is a cyclic permutation $[\phi]$ such that $H_{[\phi]} \cong H$, then clearly $H$ is self transitive.  

Conversely, let $H$ be a self transitive oriented $3$--hypergraph with $n$ vertices. We proceed by induction on the number of vertices, in order to prove that $H$ is an orientedcyclic permutation $3$--hypergraph.  The statement is obvious for $n=3$.

Assume $n\geq 4$, and label the vertices of $H$ by $V(H)=\{1,2, ...,n\}$.  Consider the $3$--hypergraph obtained from $H$ by removing the vertex $n$, this is $H\setminus \{n\}$. Since $H\setminus \{n\}$ is a self transitive $3$--hypergraph, thus by induction hypothesis there is a cyclic permutation $[\varphi]$ of $[n-1]$, such that $(H\setminus \{n\})\cong H_{[\varphi]}$.

Consider now, the oriented graph $link_H (n)$. By Lemma~\ref{link}, $link_H (n)$ is a self transitive oriented graph, and therefore there is a linear permutation $\psi$, such that $link_H (n)\cong H_{\psi}$.
In the remainder of the proof we will show that the cyclic equivalent class of $\psi$ is precisely  $[\varphi]$, and that $[\varphi]$ can be extend to a cyclic permutation $[\phi]$ of $[n]$ in such a way that $H\cong H_{[\phi]}$.

\textbf{Claim 1.} $\psi \in [\varphi]$

Since $[\varphi]$ is a cyclic permutation, with out lost of generality we may assume that $\varphi(1)=\psi(1)$. Next we will prove that $\varphi(i)=\psi(i)$ for every $i\in [n-1]$. By contradiction, let $j\in [n-1]$ be the first integer such that $\varphi(j)\neq \psi(j)$. Then, the cyclic ordering of $\{\varphi(1),\varphi(j),\psi(j)\}$ induced by $[\varphi]$ is $(\varphi(1)\,\varphi(j)\,\psi(j))$. Hence, either $(\varphi(1)\,\varphi(j)\,\psi(j))\in O(H)$ or $(\varphi(1),\psi(j)\,\varphi(j))\in O(H)$. With out lost of generality we may suppose the first case,  then one of the following holds: $\varphi(1)< \varphi(j)< \psi(j)$ or $\psi(j)< \varphi(1)< \varphi(j)$ or $\varphi(j)< \psi(j)<\varphi(1)$.  Assume first $\varphi(1)< \varphi(j)< \psi(j)$,  note that $(\varphi(j)\,\psi(j))\not\in O(link_H (n))$, thus $(n\,\psi(j)\, \varphi(j))\in O(H)$, and by transitivity $(n\,\varphi(j)\,\varphi(1))\in O(H)$, which is a contradiction since $(\varphi(1)\,\varphi(j)), (\varphi(1)\,\psi(j))\in O(link_H (n))$. The other cases follows by similar arguments, so the claim holds.

Consider now the cyclic permutation 
$(\psi(1)\,\psi(2)\,...\,\psi(n-1)\,n)$ and call it $[\phi]$.

\textbf{Claim 2.} $H\cong H_{[\phi]}$.

Recall $(H\setminus \{n\})\cong H_{[\varphi]}$. Since $\psi \in [\varphi]$ then the orientations of all edges in $H\setminus \{n\}$ are induced by $[\phi]$. It remains to prove that the orientations of all edges containing $n$, are induced by $[\phi]$. To see this, note that the orientations of such edges are given by $link_H (n)$, and for every $j\in [n-1]$ it happens that $\phi^{-1}(j)<\phi^{-1}(n)=n$. Thus the orientations of all edges containing $n$ are induced by $[\phi]$, which completes the proof.

\end{proof}


\subsection{Cyclic comparability $3$--hypergraphs}

\label{sec:compa}

One of the classic results in the study of permutation graphs is their characterization in terms of comparability graphs also referred in the literature \cite{Golumbic1980} as transitively oriented graphs. The aim of this section is to  prove an equivalent result for cyclic permutations and cyclic comparability $3$--hypergraphs.

\begin{defi}
A $3$--hypergraph, $H$, is called a \emph{cyclic comparability hypergraph} if it admits a transitive orientation.
\end{defi}

\begin{lem}
\label{lem:union2}Let $H$ be a cyclic comparability $3$--hypergraph such that $\overline{H}$ is also a cyclic comparability $3$--hypergraph. Let $H_{o}$ and $\overline{H_{o}}$ be any transitive oriented $3$--hypergraphs whose underlying $3$--hypergraphs are $H$ and $\overline{H}$ respectively. Then the union $H_{o} \cup \overline{H_{o}}$ is isomorphic to $TT_n^3$.
\end{lem}

\begin{proof}

Clearly $H_{o} \cup \overline{H_{o}}$ is a complete $3$--hypergraph, thus, we only need to verify that  $H_{o} \cup \overline{H_{o}}$ is transitive. Observed, that the, transitivity in 3-hypergraphs is a local condition, therefore it is sufficient to check that the transitivity follows for every set of four vertices.

Let $F$ be any comparability $3$--hypergraph with four vertices such that $\overline{F}$ is also a comparability $3$--hypergraph.   We shall prove that $F$ is transitive by giving a cyclic order of its vertices.

If $F$ has no edges, or four edges, the transitivity follows. So we might assume both $F$ and $\overline{F}$ have two edges each, as if any of them has three edges, by the evenness condition (Lemma~\ref{eveness}), it can't be oriented transitively.

We might assume with out lost of generality that $E(F)=\{\{v_i, v_j, v_k\}, \{v_i, v_j, v_l\}\}$ and $ E(\overline{F})=\{\{v_i, v_k, v_l\}, \{v_j, v_k, v_l\}\}$. Then for the orientation of $F$ we might assume $O(F)=\{(v_i\; v_j\; v_k), (v_i\; v_j\; v_l) \}$ (otherwise we might relabel the vertices) and $O(\overline{F})=\{(v_i\; v_k\; v_l), (v_j\;v_k\; v_l)\}$ or $O(\overline{F})=\{(v_i\; v_l\; v_k), (v_j\; v_l\; v_k)\}$.

Assume $O(\overline{F})=\{(v_i\; v_k\; v_l), (v_j\;v_k\; v_l)\}$. Out of all the six possible pairings of the four edges in $H \cup \overline{H}$,  only two have opposite orders along their common pair of vertices, namely $\{(v_i\; v_j\; v_k), (v_i\; v_k\; v_l)\}$ and $\{(v_i\; v_j\; v_l), (v_j\; v_k\; v_l)\}$. By transitivity the first pairing indicates that the orientations of the two other edges are induced from the ciclic order $(v_i\; v_j\; v_k\; v_l)$. For the proof of the case, $O(\overline{F})=\{(v_i\; v_l\; v_k), (v_j\; v_l\; v_k)\}$ the argument is the same.
\end{proof}

An un-oriented $3$--hypergraph $H$ is called a \emph{cyclic permutation $3$--hypergraph} if $H$ can be oriented in such a way that the resulting oriented $3$--hypergraph is an oriented cyclic permutation $3$--hypergraph.
The following Theorem is a direct Corollary of  Lemma \ref{lem:union2} and Theorem \ref{thm:permutation}.

\begin{thm}\label{thm:comparability} 
A $3$-- hypergraph $H$ is a cyclic permutation $3$--hypergraph if and only if $H$ and $\overline{H}$ are cyclic comparability $3$--hypergraphs.
\end{thm}


\subsection{Total extendability of a certain class of cyclic orders}

It is a well known fact ~\cite{Alles1991, Megiddo1976} that, in contrast to
linear orders, not every cyclic order can be extended to a complete cyclic
order, which in our setting means that not every transitive $3$--hypergraph
is a subhypergraph of a transitive $3$--hypertournament.

The decision whether a cyclic order is totally extendable is known to be
NP-complete, and examples of cyclic orders that are not extendable are well
known~\cite{Alles1991, Megiddo1976}. In ~\cite{Alles1991} the authors
exhibit some classes of cyclic orders which are totally extendable.

As a direct consequence of Lemma \ref{lem:union2} we can state the following
theorem stating a sufficient condition for the extendability of cyclic
orders.


\begin{thm}
\label{thm:main_nes} Let $T$ be a partial cyclic order on $X$, if the complement relation $\overline{%
T}$ is a cyclic order then $T$ is totally extendable.
\end{thm}

\begin{proof}
Let $H_{T}$ and $H_{\overline T}$ be the transitive oriented
$3$--hypergraphs associated to the cyclic orders $T$ and $\overline
T$, respectively. Then by Lemma \ref{lem:union2}, $H_{T} \cup
H_{\overline T} \cong TT_{n}^{3}$, so that there is a cyclic
ordering of the elements in $X$ that induces all orientations of
edges in $H_{T}$, this is $T$ is totally extendable.
\end{proof}

\section{Conclusions}

The proofs of all theorems on oriented $3$-hypergraphs in this paper follow naturally from the definition of transitivity and, not surprisingly, all concepts defined for oriented $3$-hypergraphs are reminiscent of the corresponding notions for oriented graphs; such evidence allows us to suggest that perhaps many parts of the theory of transitive graphs can be extended to transitive $3$-hypergraphs.\\
In particular, we are interested in the extension of the concept of perfection for hypergraphs. That was indeed the reason why we started the exploration on the subject of transitive $3$-hypergraphs with the extension of the concepts of comparability graphs and permutation graphs, which are two important classes of perfect graphs \cite{Golumbic1980}.

\medskip

\noindent The authors will like to thank the support from Centro de
Innovaci\'on Matem\'atica A.C.

\end{document}